\documentclass[10pt]{amsart}
\usepackage{hyperref}
\usepackage{amsmath,amssymb,latexsym,amsthm,mathrsfs}
\usepackage{enumerate}
\usepackage{bm,cite}
\usepackage{graphicx,color}
\newtheorem{thm}{Theorem}
\newtheorem{lemma}[thm]{Lemma}
\newtheorem{prop}[thm]{Proposition}
\newtheorem{defi}[thm]{Definition}

\newtheorem{coro}[thm]{Corollary}

\def\diver{\text{div}\,}
\def\curl{\text{curl}\,}
\def\supp{\text{supp}\,}

\title[an inverse problem in electro-seismic imaging]{on an inverse source
problem for the biot equations in electro-seismic imaging}

\date{}

\begin{document}

\author[Y. Gao]{Yixian Gao}
\address{School of Mathematics and Statistics, Center for Mathematics and
Interdisciplinary Sciences, Northeast Normal University, Changchun, Jilin
130024, P. R. China}
\email{gaoyx643@nenu.edu.cn}

\author[P. Li]{Peijun Li}
\address{Department of Mathematics, Purdue University, West Lafayette, IN
47907, USA.}
\email{lipeijun@math.purdue.edu}

\author[Y. Yang]{Yang Yang}
\address{Department of Computational Mathematics, Science and Engineering,
Michigan State University, East Lansing, MI 48824, USA.}
\email{yangy5@msu.edu}

\thanks{The research of YG was supported in part by NSFC grant 11671071,
JLSTDP 20160520094JH and FRFCU2412017FZ005. The research of YY was partly
supported by NSF Grant DMS-1715178, an AMS Simons travel grant, and a start-up
fund from Michigan State University.}

\maketitle

\begin{abstract}
Electro-seismic imaging is a novel hybrid imaging modality in geophysical
exploration. This paper concerns an inverse source problem for Biot's
equations that arise in electro-seismic imaging. Using the time reversal method,
we derive an explicit reconstruction formula, which immediately gives the 
uniqueness and stability of the reconstructed solution.
\end{abstract}

\section{Introduction}

Electro-Seismic (ES) imaging and Seismo-Electric (SE) imaging are
emerging modalities in geophysical exploration. Compared with traditional
modalities, they provide images of high accuracy with low cost, which has led
to wide applications in locating groundwater aquifers and petroleum hydrocarbon
reservoir. The underlying physical phenomena of ES and SE imaging are known as
electro-seismic conversion and seismo-electric conversion,
respectively. These conversions usually occur in a porous medium and have been
used as bore-hole logging and cross-hole logging tools
\cite{ARJMK2012, DBKH2009, HGH2007, Zhu2005}.

Both ES and SE imaging involve conversions of electromagnetic energy and seismic
energy. The governing equations were derived by Pride
\cite{P1994}:
\begin{align}
\nabla \times E = -\mu\partial_t H \vspace{1ex}, \label{eqn:max1}\\
\nabla \times H = (\epsilon \partial_t + \sigma) E + L(-\nabla p - \rho_{12}
\partial^2_t \mathbf{u^s}), \label{eqn:max2}\\
\rho_{11} \partial^2_t \mathbf{u^s} + \rho_{12} \partial^2_t \mathbf{u^f} =
\diver{\tau},\label{eqn:biot1}\\
\rho_{12} \partial^2_t \mathbf{u^s} + \rho_{22} \partial^2_t \mathbf{u^f} +
\frac{\eta}{\kappa} \partial_t \mathbf{u^f} + \nabla p = \frac{\eta}{\kappa} LE,
\label{eqn:biot2}\\
\tau = (\lambda \, \diver{\mathbf{u^s}} + q \, \diver{\mathbf{u^f}}) I_3 + \mu
(\nabla \mathbf{u^s} + (\nabla \mathbf{u^s})^T), \label{eqn:biot3}\\
-p = q \, \diver{\mathbf{u^s}} + r\, \diver{\mathbf{u^f}}, \label{eqn:biot4}
\end{align}
where $E$ is the electric field, $H$ is the magnetic field, $\mathbf{u}^s$ is
the solid displacement, $\mathbf{u}^f$ is the fluid displacement, $\mu$ is the
magnetic permeability, $\epsilon$ is the dielectric constant, $\sigma$ is the
conductivity, $L$ is the electro-kinetic mobility parameter, $\rho_{11}$ is a
linearly combined density of the solid and the fluid, $\rho_{12}$ is the
density of the pore fluid, $\rho_{22}$ is the mass coupling coefficient,
$\kappa$ is the fluid flow permeability, $\eta$ is the viscosity of pore fluid,
$\lambda$ and $\mu$ are the Lam\'{e} parameters, $q$ and $r$ are the Biot moduli
parameters, $\tau$ is the bulk stress tensor, $p$ is the pore pressure, $I_3$ is
the $3\times 3$ identity matrix. Equations \eqref{eqn:max1}--\eqref{eqn:max2}
are Maxwell's equations, modeling the electromagnetic wave propagation.
Equations \eqref{eqn:biot1}--\eqref{eqn:biot4} are Biot's equations,
describing the seismic wave propagation in the porous medium \cite{B1956,
B1956_1}.

In this work, we focus on the ES imaging. The inverse problem in ES
imaging concerns recovery of the physical parameters in 
\eqref{eqn:max1}--\eqref{eqn:biot4} from boundary measurement of the
displacements $\mathbf{u^s}$ and $\mathbf{u^f}$. This problem can be studied in
two mutually relevant steps. The first step concerns
an inverse source problem for the Biot equations
\eqref{eqn:biot1}--\eqref{eqn:biot4} to recover the coupling term
$\frac{\eta}{\kappa} LE$. The second step utilizes
this term as internal measurement to retrieve physical parameters in the
Maxwell equations \eqref{eqn:max1}--\eqref{eqn:max2}. The second step has been
considered in \cite{CY2013, CdH2014}.

The first step is not well understood mathematically. In an
unpublished work of Chen and de Hoop \cite{C2015}, they suggested using
Gassmann's approximation \cite{G1951} to reduce Biot's equations to the
elastic equation and then applying the result in \cite{T2012}. This approach has
the limitation that Gassmann's approximation is valid only when the fluid
permeability is small and the wave frequency is sufficiently low. In
\cite{BY2013}, an inverse source problem was investigated in a different context
by assuming access to internal measurement rather than boundary measurement.

Our goal is to demonstrate a general approach to the first step and complete
the two-step approach towards the coupled physics inverse problem in ES
imaging. We study the inverse source problem from boundary measurement of
$\mathbf{u^s}$ and $\mathbf{u^f}$ directly for Biot's equations, without
resorting to Gassmann's approximation or any internal measurement. We derive an
explicit reconstruction formula in terms of a Neumann series. Uniqueness and
stability of the reconstructed solution are immediate consequences of the
explicit formula.

\section{Problem formulation and main result}

We make two simplifications to the Biot equations
\eqref{eqn:biot1}--\eqref{eqn:biot4}. First we ignore the attenuation
term $\frac{\eta}{\kappa}\partial_t\mathbf{u}^{\bf f}$ in \eqref{eqn:biot2}.
This simplification is not essential since the first order time derivative can
be handled by adding a perturbation argument to our proof. Second, we assume
that the source term $\frac{\eta}{\kappa}LE$ is instantaneous in time and has
the form
\[
\frac{\eta(x)}{\kappa(x)}L(x)E(t,x)=\delta(t)f(x),
\]
where $f$ is an unknown source function. This assumption is necessary to gain
uniqueness in the inverse problem since our boundary
measurement has merely three degrees of freedom and cannot be expected to
recover four degrees uniquely. 

Let us rewrite Biot's equations (without attenuation) in different forms for
subsequent analysis. It follows from \eqref{eqn:biot3}--\eqref{eqn:biot4} that
we have 
\begin{align*}
\diver{\tau} =  \Delta_{\mu,\lambda}\mathbf{u^s} +
\nabla(q\,\diver{\mathbf{u^f}}), \quad
\nabla p =  -\nabla (q \, \diver{\mathbf{u^s}} ) - \nabla (r\,
\diver{\mathbf{u^f}} ),
\end{align*}
where the elastic operator $\Delta_{\mu,\lambda}$ is defined by
\begin{equation} \label{elasticop}
\begin{array}{rl}
\Delta_{\mu,\lambda} \mathbf{u^s} := & \diver{(\mu(\nabla \mathbf{u^s} + (\nabla
\mathbf{u^s})^T))} + \nabla(\lambda\diver{\mathbf{u^s}}), \\
= & \mu\Delta \mathbf{u^s} + (\mu+\lambda)\nabla(\diver{\mathbf{u^s}}) +
(\diver{\mathbf{u^s}})\nabla\lambda + (\nabla \mathbf{u^s} + (\nabla
\mathbf{u^s})^T)\nabla\mu.
\end{array}
\end{equation}
Then the Biot equations (without attenuation) can be written as
\begin{equation} \label{newBiot}
\left\{
\begin{array}{rcl}
\rho_{11} \partial^2_t \mathbf{u^s} + \rho_{12} \partial^2_t \mathbf{u^f} -
\Delta_{\mu,\lambda}\mathbf{u^s} - \nabla(q\,\diver{\mathbf{u^f}} ) & = & 0
\vspace{1ex}\\
\rho_{12} \partial^2_t \mathbf{u^s} + \rho_{22} \partial^2_t \mathbf{u^f} -
\nabla(q\,\diver{\mathbf{u^s}}) - \nabla(r\,\diver{\mathbf{u^f}}) & = &
\delta(t)f(x).
\end{array}
\right.
\end{equation}

One can further write the system \eqref{newBiot} into a matrix equation. Let
\begin{equation} \label{MPD}
M = 
\left(
\begin{array}{cc}
\rho_{11} I_3 & \rho_{12} I_3 \\
\rho_{12} I_3 & \rho_{22} I_3 
\end{array}
\right),
\quad\quad
P(D) =
\left(
\begin{array}{cc}
-\Delta_{\mu,\lambda}\cdot & -\nabla(q\,\diver{\cdot} ) \\
-\nabla(q\,\diver{\cdot} ) & -\nabla(r\,\diver{\cdot} ) 
\end{array}
\right).
\end{equation}
The matrix form of \eqref{newBiot} is
\[
M \partial^2_t \mathbf{u} + P(D) \mathbf{u} = \delta(t)\mathbf{f}.
\]
which, by Duhamel's principle, is equivalent to the initial value problem:
\begin{equation} \label{forward}
\left\{
\begin{array}{rcl}
M \partial^2_t \mathbf{u} + P(D) \mathbf{u} & = & \mathbf{0} \quad\quad\quad
\text{ in } \mathbb{R}\times\mathbb{R}^3, \\
\mathbf{u}|_{t=0} & = & \mathbf{f(x)}, \\
\partial_t \mathbf{u}|_{t=0} & = & \mathbf{0}. 
\end{array}
\right.
\end{equation}

Hereafter, a vector with $6$ components is written in bold
face such as $\mathbf{u}$. It is often split into two $3$-vectors such as
$\mathbf{u}=(\mathbf{u^s}, \mathbf{u^f})\in\mathbb{R}^3\times\mathbb{R}^3$. The
names of the two $3$-vectors are chosen to be consistent with the splitting for
solutions of the Biot equations. Following this notation, the right hand side
is defined as $\mathbf{f}=(\mathbf{f^s},\mathbf{f^f}):=(0,f)$. Hence
we will not refer to the specific form any longer but consider a general
$\mathbf{f}$. Well-posedness of the initial value problem \eqref{forward} can be
proved in a similar way as \cite[Theorem 1.1]{BY2013}.

The following hypothesis are crucial and assumed throughout the paper.

(\textbf{H1}): The functions $\rho_{11}$, $\rho_{12}$,
$\rho_{22}$, $\mu$, $\lambda$, $q$, $r$ are bounded from below by
a constant, say $c>0$, and have bounded derivatives. 

(\textbf{H2}): $\rho_{11}\rho_{22}-\rho^2_{12}>0~ \text{ in }
\mathbb{R}^3. $

(\textbf{H3}): $\lambda r-q^2>0 ~\text{ in } \mathbb{R}^3.$

\noindent Here (\textbf{H1}) states that all the physical parameters are
positive and sufficiently smooth. (\textbf{H2}) ensures that the matrix $M$ is
positive definite. (\textbf{H3}) guarantees that the energy functional defined
in Section \ref{sec:funcspace} is non-negative.

In ES imaging, the measurement is the solid displacement $\mathbf{u^s}$ and the
fluid displacement $\mathbf{u^f}$ on the boundary of a domain-of-interest
$\Omega$ which is a smooth bounded open subset in $\mathbb{R}^{3}$. Introduce
the source-to-measurement operator $\bm{\Lambda}$ as follows:
\[
\bm{\Lambda} \mathbf{f} := \mathbf{u}|_{[0,T]\times\partial\Omega}
\]
where $\mathbf{u}$ is the solution of the initial value problem \eqref{forward}
and $T>0$ represents a duration of the measurement. Note that the initial
displacement occurs inside the domain-of-interest, which means that $\mathbf{f}$
is supported in the interior of $\Omega$. We are interested in recovering
information on $\mathbf{f}$ from the boundary measurement
$\bm{\Lambda}\mathbf{f}$.

A quick look at the equations \eqref{newBiot} suggests that there is a certain
gauge transform for the recovery of $\mathbf{u^f}$. In fact, adding
to $\mathbf{u^f}$ by any vector field $g$ that is divergence free (i.e,
$\diver{g}=0$) does not affect \eqref{newBiot}. This is a reflection of the
simple fact that $\mathbf{u^f}$ appears in the equations only in the form of
$\diver{\mathbf{u^f}}$. Taking such gauge into account, we raise the following
question which is the central topic of this paper.

\textbf{Inverse source problem:} Suppose that $M$ and $P(D)$ are known and
satisfy the hypotheses (\textbf{H1})-(\textbf{H3}), can one reconstruct the
initial source $\mathbf{f}=(\mathbf{f^s}, \mathbf{f^f})$, compactly supported in
$\Omega$, from the boundary measurement $\bm{\Lambda}\mathbf{f}$ up to a pair of
vector fields $(0,g)$ with $g$ divergence free?

We give an affirmative answer to the question, including a reconstruction
formula. Our main result can be summarized as follows. A rigorous restatement of
this theorem is given in Theorem \ref{thm:main}. 

\begin{thm}
Under appropriate assumptions, $\mathbf{f}$ can be uniquely and stably
reconstructed from $\bm{\Lambda}\mathbf{f}$ by a convergent Neumann series, up
to a pair of vector fields $(0,g)$ with $g$ divergence free.
\end{thm}

Our proof is based on the modified time reversal method proposed by Stefanov
and Uhlmann \cite{SU2009} for the thermo-acoustic tomography
(TAT) \cite{KK2008}. TAT is a hybrid modality in medical imaging where
optical or electromagnetic waves are exerted to trigger ultrasound wave in
tissue through thermo-elastic conversion. Conventional time
reversal method is known to provide an approximate reconstruction of the source
\cite{HKN2008}. Stefanov and Uhlmann improved it and obtained an accurate
reconstruction of the source by a Neumann series \cite{SU2009}  which was
numerically implemented in \cite{CLQ2014, QSUZ2011}. This improved time reversal
method has since been adapted and generalized to many other models
\cite{dHT2015, H2013, KN2017, P2016, SU2011, T2012}. We refer to \cite{B2012}
for a survey on more hybrid modalities in the context of medical imaging.

A slight variation of our approach can be used to reconstruct
$\mathbf{f(x)}$ when the source term takes the form
$L_0\frac{\eta}{\kappa}E=\delta'(t)\mathbf{f(x)}$. Indeed, it follows
from Duhamel's principle that we have 
\[
\left\{
\begin{array}{rcl}
M \partial^2_t \mathbf{u} + P(D) \mathbf{u} & = & \mathbf{0} \quad\quad\quad
\text{ in } \mathbb{R}\times\mathbb{R}^3, \\
\mathbf{u}|_{t=0} & = & \mathbf{0}, \\
\partial_t \mathbf{u}|_{t=0} & = & \mathbf{f(x)}.
\end{array}
\right.
\]
If $\mathbf{u}$ solves this problem, the $\partial_t\mathbf{u}$ solves
\eqref{forward} by \cite{SU2013}. Hence one can
first reconstruct $\partial_t\mathbf{u}$ using our approach and then integrate
to obtain $\mathbf{u}$.

The rest of the paper is organized as follows. In Section \ref{sec:system} we
convert Biot's equations (without attenuation) to two hyperbolic systems.
These systems are used in Section \ref{sec:FSPUC} to prove finite speed of
propagation and unique continuation results. Section \ref{sec:funcspace} is
devoted to discussion of function spaces. The main theorem is stated
and proved in Section \ref{sec:mainthm}.

\section{Biot's equations} \label{sec:system}

In this section we transform Biot's equations into two hyperbolic systems: a
principally scalar system and a symmetric hyperbolic system.

\subsection{The principally scalar system}

Let us start with the principally scalar system. Recall the definition of a
principally scalar system \cite{EINT2002}. For a function $a=a(x)$, define the
scalar wave operator $\Box_a:=a\partial^2_t-\Delta$. 
\begin{defi}
A principally scalar system refers to
\begin{equation} \label{pss}
\Box_{a_j}u_j + b_j(t,x,\nabla u) + c_j(t,x,u) = f_j,\quad\quad j=1,\dots,m.
\end{equation}  
Here $u=(u_1,\dots,u_m)$, $a_j=a_j(x)\in C^1$ are real-valued functions,
$b_j$ and $c_j$ are linear functions with $L^\infty_{loc}$-coefficients of
$\nabla u$ and $u$. 
\end{defi}

The system is called principally scalar as the principal part of each equation
is a scalar wave operator. We shall see in Section \ref{sec:FSPUC} 
that a principally scalar systems can be uniquely continued towards inside from
proper boundary data. This unique continuation property is crucial to obtain
Proposition \ref{thm:BiotUC}.

We can write the equation in \eqref{forward} as a principally scalar system
by following the procedures in \cite{BY2013}.
Let $\rho(x)=\rho_{11}(x)\rho_{22}(x)-\rho^2_{12}(x)$. We have $\rho>0$ by
(\textbf{H2}) and
$$M^{-1} = 
\frac{1}{\rho}
\left(
\begin{array}{cc}
\rho_{22} I_3 & -\rho_{12} I_3 \\
-\rho_{12} I_3 & \rho_{11} I_3 
\end{array}
\right).
$$
Multiplying the equation in \eqref{forward} by $M^{-1}$ to get
$$
\partial^2_t
\left(
\begin{array}{c}
 \mathbf{u^s} \\
 \mathbf{u^f}
\end{array}
\right)
+ M^{-1}
P(D)\left(
\begin{array}{c}
 \mathbf{u^s} \\
 \mathbf{u^f}
\end{array}
\right)
=
\left(
\begin{array}{c}
 0 \\
 0
\end{array}
\right),
$$
where
$$
M^{-1}
P(D)=
\frac{1}{\rho}
\left(
\begin{array}{cc}
-\rho_{22}\Delta_{\mu,\lambda}\cdot + \rho_{12}\nabla(q\,\diver{\cdot}) &
-\rho_{22}\nabla(q\,\diver{\cdot})+\rho_{12}\nabla(r\,\diver{\cdot}) \\
\rho_{12}\Delta_{\mu,\lambda}\cdot - \rho_{11}\nabla(q\,\diver{\cdot}) &
\rho_{12}\nabla(q\,\diver{\cdot})-\rho_{11}\nabla(r\,\diver{\cdot}).
\end{array}
\right).
$$

Now, we write the matrix equation as a system of equations and move all the
first and zeroth order derivatives to the right hand side to get
\begin{equation} \label{newBiot1}
\left\{
\begin{array}{rcl}
\partial^2_t \mathbf{u^s} - \mu_1 \Delta \mathbf{u^s} -
(\mu_1+\lambda_1)\nabla\diver\mathbf{u^s} - q_1\nabla\diver\mathbf{u^f} & = &
\mathscr{P}(\mathbf{u^s}, \mathbf{u^f}),\\
\partial^2_t \mathbf{u^f} + \mu_2\Delta\mathbf{u^s} -r_2\nabla(\diver
\mathbf{u^f}) - q_2\nabla(\diver \mathbf{u^s}) & = & \mathscr{Q}(\mathbf{u^s},
\mathbf{u^f}),
\end{array}
\right.
\end{equation}
where the new coefficients are
$$\mu_1=\rho^{-1}\rho_{22}\mu,\quad
\lambda_1=\rho^{-1}(\rho_{22}\lambda-\rho_{12}q), \quad
q_1=\rho^{-1}(\rho_{22}q-\rho_{12}r),$$
$$\mu_2=\rho^{-1}\rho_{12}\mu,\quad r_2=\rho^{-1}(\rho_{11}r-\rho_{12}q),\quad
q_2=\rho^{-1}(\rho_{11}q-\rho_{12}(\mu+\lambda)).$$
Here and below, the script letters $\mathscr{P}, \mathscr{Q}, \mathscr{R},
\dots$ denote various first-order linear differential operators.

A straightforward calculation shows that first order derivatives in
$\mathscr{P}(\mathbf{u^s}, \mathbf{u^f})$ and $\mathscr{Q}(\mathbf{u^s},
\mathbf{u^f})$ appear in terms of $\nabla\mathbf{u^s}$,
$(\nabla\mathbf{u^s})^T$, $\diver\mathbf{u^s}$, and $\diver\mathbf{u^f}$. The
right hand side can be simplified by the substitutions as indicated below.

Introduce the substitutions:
\begin{equation} \label{vs}
v^s=\diver\mathbf{u^s}, \quad v^f=\diver\mathbf{u^f},\quad
\mathbf{v^s}=\curl\mathbf{u^s}.
\end{equation} 
Applying $\diver$ to \eqref{newBiot1} and utilizing the relations 
$$
\diver\nabla\mathbf{u^s} = \Delta\mathbf{u^s} = \nabla(\diver\mathbf{u^s}) -
\curl\;\curl\;\mathbf{u^s}, \quad\diver(\nabla \mathbf{u^s})^T =
\nabla(\diver \mathbf{u^s}),
$$
we get
\[
\left\{
\begin{array}{rcl}
\partial^2_t v^s - a_{11} \Delta v^s - a_{12} \Delta v^f & = & \mathscr{R}(v^s,v^f,\mathbf{u^s},\mathbf{v^s}) \\
\partial^2_t v^f - a_{21}\Delta v^s - a_{22}\Delta v^f & = & \mathscr{S}(v^s,v^f,\mathbf{u^s},\mathbf{v^s});
\end{array}
\right.
\]
or equivalently
$$
\partial^2_t
\left(
\begin{array}{c}
 v^s \\
 v^f
\end{array}
\right)
-
A(x)\Delta
\left(
\begin{array}{c}
 v^s \\
 v^f
\end{array}
\right)
= \mathscr{L}(v^s,v^f,\mathbf{u^s},\mathbf{v^s}),
$$
where the coefficient $A(x)$ is
\[
A(x) :=
\left(
\begin{array}{cc}
 a_{11} & a_{12} \\
 a_{21} & a_{22}
\end{array}
\right)
=
\frac{1}{\rho}
\left(
\begin{array}{cc}
 \rho_{22} & -\rho_{12} \\
 -\rho_{12} & \rho_{11}
\end{array}
\right).
\left(
\begin{array}{cc}
 2\mu+\lambda & q \\
 q & r
\end{array}
\right).
\]
The hypotheses (\textbf{H1})--(\textbf{H3}) ensures that $A$ is a symmetric
positive definite matrix. Let $a_1, a_2$ be its eigenvalues,  
then there exists a non-singular matrix $Q(x)$ such that
$$
(Q^{-1}AQ)(x) = \text{Diag}\,(a_1,a_2)(x).
$$
Making the change of variable 
\[
(\tilde{v}^s,\tilde{v}^f)=Q^{-1}(v^s,v^f)
\]
yields that $(\tilde{v}^s,\tilde{v}^f)$ solves
$$
\left\{
\begin{array}{rcl}
\partial^2_t \tilde{v}^s - a_1 \Delta \tilde{v}^s & = &
\mathscr{M}(\tilde{v}^s,\tilde{v}^f,\mathbf{u^s},\mathbf{v^s}), \\
\partial^2_t \tilde{v}^f - a_2 \Delta \tilde{v}^f & = &
\mathscr{N}(\tilde{v}^s,\tilde{v}^f,\mathbf{u^s},\mathbf{v^s}).
\end{array}
\right.
$$

Applying $\curl$ to the first equation of \eqref{newBiot1} and using
$\curl\nabla\mathbf{u}=0$ gives
$$
\partial^2_t \mathbf{v^s} - \mu_1\Delta \mathbf{v^s} = \mathscr{T}(\tilde{v}^s,\tilde{v}^f,\mathbf{u^s},\mathbf{v^s}).
$$
The first equation of \eqref{newBiot1} can be written as
$$
\partial^2_t \mathbf{u^s} - \mu_1\Delta \mathbf{u^s} = \mathscr{U}(\tilde{v}^s,\tilde{v}^f,\mathbf{u^s},\mathbf{v^s}).
$$
Thus we obtain the following principally scalar system in the variables
$(\tilde{v}^s, \tilde{v}^f, \mathbf{u^s},\mathbf{v^s})$:
\begin{equation} \label{pssystem}
\left\{
\begin{array}{rcl}
\Box_{\frac{1}{a_1}} \tilde{v}^s -
\frac{1}{a_1}\mathscr{M}(\tilde{v}^s,\tilde{v}^f,\mathbf{u^s},\mathbf{v^s}) & =
& 0,\\
\Box_{\frac{1}{a_2}} \tilde{v}^f -
\frac{1}{a_2}\mathscr{N}(\tilde{v}^s,\tilde{v}^f,\mathbf{u^s},\mathbf{v^s}) & =
& 0,\\
\Box_{\frac{1}{\mu_1}} \mathbf{u^s} -
\frac{1}{\mu_1}\mathscr{T}(\tilde{v}^s,\tilde{v}^f,\mathbf{u^s},\mathbf{v^s}) &
= & 0,\\
\Box_{\frac{1}{\mu_1}} \mathbf{u^f} -
\frac{1}{\mu_1}\mathscr{U}(\tilde{v}^s,\tilde{v}^f,\mathbf{u^s},\mathbf{v^s}) &
= & 0. 
\end{array}
\right.
\end{equation}
This is the desired principally scalar system. Note that $a_1, a_2, \mu_1$ are
smooth and strictly positive by (\textbf{H1}), hence their reciprocals exist and
are smooth as well. Another observation is that $(\tilde{v}^s,\tilde{v}^f)=0$ if
and only if $(v^s,v^f)=0$.

\subsection{The symmetric hyperbolic system}

We proceed to write the principally scalar system \eqref{pssystem} as a first
order symmetric hyperbolic system \cite{R2012}, which will be exploited to
show that the Biot equations have finite speed of propagation. We restrict to
the case where the coefficients are time-independent matrices.

For $x\in\mathbb{R}^3$, let $A_0(x), \dots, A_{4}(x)$ be matrix-valued
functions. Denote by $\mathcal{L}$ the partial differential operator 
\begin{equation} \label{symhyper}
\mathcal{L}(x,\partial_t,\partial_x)= A_0(x)\partial_t + \sum^3_{j=1} A_j(x)
\partial_{x_j} + A_{4}(x),
\end{equation}
where the coefficient matrices $A_0, \dots, A_4$ are assumed to have uniformly
bounded derivatives, i.e., 
$$\sup_{x\in\mathbb{R}^3}\left\| \partial^\alpha_{x}\left(
A_0(x),\dots, A_{4}(x)\right) \right\| < \infty\quad \text{ for any
multi-index } \alpha.$$

\begin{defi}
$\mathcal{L}$ is called symmetric hyperbolic if the following conditions hold:
\begin{enumerate}[(i)]
\item $A_0(x), \dots, A_3(x)$ are symmetric;
\item $A_0$ is strictly positive, i.e., there is a constant $C>0$ such that for
all $x\in\mathbb{R}^3$,
$$A_0(x) \geq CI,$$
where $I$ is the identity matrix.
\end{enumerate}
\end{defi}

Our goal is to convert the principally scalar system \eqref{pssystem} to a
symmetric hyperbolic system. We begin with a scalar wave equation to demonstrate
the procedures. The principally scalar system is treated afterwards.

Let $v(x)$ be a scalar function defined in $\mathbb{R}^3$ and satisfy the equation
\begin{equation} \label{scalar}
\partial^2_t v - c(x)\Delta v = \mathscr{L}(\partial_{x_1}v, \partial_{x_2}v,
\partial_{x_3}v, v),
\end{equation}
where $\mathscr{L}(\partial_{x_1}v, \partial_{x_2}v, \partial_{x_3}v, v)$ is
linear in each derivative. We define a vector $V(v)$: 
$$
V(v)=(V_0,V_1, V_2, V_3, V_4):=(\partial_t v, \partial_{x_1}v, \partial_{x_2}v, \partial_{x_3}v, v).
$$
The scalar equation \eqref{scalar} can be written in terms of $V_0, \dots,
V_{4}$ as follows:
$$ 
\left\{
\begin{array}{rcl}
\partial_t V_0 - a_1 \partial_{x_1} V_1 - a_1 \partial_{x_2} V_2 - a_1
\partial_{x_3} V_3 - \mathscr{L}(V_1, \dots, V_{4}) & = & 0, \\
a_1 \partial_t V_1 - a_1 \partial_{x_1} V_0 & = & 0, \\
a_1 \partial_t V_2 - a_1 \partial_{x_2} V_0 & = & 0, \\
a_1 \partial_t V_3 - a_1 \partial_{x_3} V_0 & = & 0, \\
a_1 \partial_t V_{4} - a_1 V_0 & = & 0. 
\end{array}
\right.
$$
In the matrix form, this system reads
$$
B_0(x) \partial_t V + B_1(x) \partial_{x_1} V + B_2(x) \partial_{x_2} V + B_3(x) \partial_{x_3} V + B_{4}(x) V = 0,
$$
where $B_0, \dots, B_{4}$ are $5 \times 5$ matrices. More explicitly, 
$B_0=\text{Diag}(1, a_1, a_1, a_1, a_1)$, $B_{4}$ is determined by the
concrete form of $\mathscr{L}$, and the other matrices are
$$
B_i(x) := 
\left\{
\begin{array}{ll}
-a_1 & \text{ for } (1,i+1) \text{ and } (i+1,1) \text{ entry} \\
0 & \text{ for other entries }
\end{array}
\right., \quad
i=1, 2, 3.
$$
Note this is a symmetric hyperbolic system since $B_0, \dots, B_3$ are
symmetric matrices and $B_0$ is strictly positive.

Now we turn to the principally scalar system \eqref{pssystem}. As each equation
in the system takes the form \eqref{scalar}, we define a vector $U$, which has
$40$ components and is obtained by juxtaposing $V(v)$ with $v$ replaced
successively by the components of
$(\tilde{v}^s,\tilde{v}^f,\mathbf{u^s},\mathbf{v^s})$, i.e.,
$$
U=(U_1,\dots,U_{40}):=(V(\tilde{v}^s), V(\tilde{v}^f), V(\mathbf{u^s_1}), V(\mathbf{u^s_2}), V(\mathbf{u^s_3}), V(\mathbf{v^s_1}), V(\mathbf{v^s_2}), V(\mathbf{v^s_3})).
$$
Since the principal part of each equation in \eqref{pssystem} is uncoupled, one
can write \eqref{pssystem} as a first order system in a similar manner:
\begin{equation} \label{Biotsymhyper}
A_0(x) \partial_t U + A_1(x) \partial_{x_1} U + A_2(x) \partial_{x_2} U + A_3(x) \partial_{x_3} U + A_{4}(x) U = 0.
\end{equation}
Here each $A_i$ is a $40 \times 40$ matrix, $A_0$ represents the diagonal matrix
$$
\text{Diag}\, ( (1,a_1, a_1, a_1, a_1), (1,a_2, a_2, a_2, a_2), (1,\mu_1, \mu_1, \mu_1, \mu_1), \dots, (1,\mu_1, \mu_1, \mu_1, \mu_1) )
$$
which is strictly positive, and $A_1, A_2, A_3$ are symmetric. This is the
desired symmetric hyperbolic system that is equivalent to the principally scalar
system \eqref{pssystem} and the Biot equations in \eqref{forward}.

\section{Finite speed of propagation and unique continuation} \label{sec:FSPUC}

We derive some results for the Biot equations regarding finite speed of
propagation and unique continuation. It is crucial to have the symmetric
hyperbolic system \eqref{Biotsymhyper} and the principally scalar system
\eqref{pssystem}.

\subsection{Finite speed of propagation}

Let $\mathcal{L}$ be the symmetric hyperbolic operator defined as in
\eqref{symhyper}. For any $\xi\in\mathbb{R}^3 \backslash\{0\}$, let
$$\Lambda(\xi)=\inf\{\ell:
A_0(x)^{-\frac{1}{2}}\left(\sum^{3}_{j=1}A_j(x)\xi_j\right)A_0(x)^{-\frac{1}{2}}
\leq \ell I\},$$
which is the smallest upper bound of the eigenvalues for 
$A_0^{-\frac{1}{2}}\left(\sum^{3}_{j=1}A_j\xi_j\right)A_0^{-\frac{1}{2}
}$.
We state the following result, which shows that the solution of a symmetric
hyperbolic system has finite speed of propagation.

\begin{prop}{\cite[Theorem 2.3.2]{R2012}} \label{prop:FSP}
Suppose $s\in\mathbb{R}$ and $u\in C([0,\infty);H^s(\mathbb{R}^3))$ satisfies a
symmetric hyperbolic system $\mathcal{L}u=0$. If 
$$\supp u(0,x)\subset \{x\in\mathbb{R}^3: x\cdot\xi\leq 0\},$$
then for $t\geq 0$,
$$\supp u(t,x)\subset \{x\in\mathbb{R}^3: x\cdot\xi\leq \Lambda(\xi)t\}.$$
In particular, set $c_{max}:=\displaystyle\max_{|\xi|=1} \Lambda(\xi)$, if
$$\supp u(0,x)\subset \{x\in\mathbb{R}^3: |x|\leq R\},$$
then
$$\supp u(t,x)\subset \{x\in\mathbb{R}^3: |x|\leq R+c_{max}|t|\}.$$
\end{prop}
\noindent Here $|\xi|$ denotes the Euclidean norm of $\xi$. If $u$ is a vector,
$\supp{u}$ stands for the union of the supports of its components.

Using this proposition, one can deduce finite speed of propagation for
the solutions of Biot's equation. Let $c_{max}$ be the number in
Proposition \ref{prop:FSP} with $\mathcal{L}$ defined on the left hand side of
\eqref{Biotsymhyper}.

\begin{coro} \label{thm:BiotFSP}
Let $\mathbf{u}=(\mathbf{u^s},\mathbf{u^f})$ solve $M \partial^2_t \mathbf{u} + P(D) \mathbf{u} = \mathbf{0}$. If
$$\supp{(\mathbf{u^s}(0,\cdot), \diver{\mathbf{u^f}}(0,\cdot))}\subset \{x\in\mathbb{R}^3: |x|\leq r\},$$
then
$$\supp{(\mathbf{u^s}(t,\cdot), \diver\mathbf{u^f}(t,\cdot))}\subset \{x\in\mathbb{R}^3: |x|\leq r+c_{max}|t|\}.$$
\end{coro}

\begin{proof}
Given $(\mathbf{u^s},\mathbf{u^f})$, one can construct the functions
$(\tilde{v}^s, \tilde{v}^f, \mathbf{u^s},\mathbf{v^s})$ as in \eqref{vs}, and
further the solution $U$ of \eqref{Biotsymhyper}. The proof is
completed by applying Proposition \ref{prop:FSP} to the symmetric
hyperbolic system \eqref{Biotsymhyper}. 
\end{proof}

\subsection{Unique continuation}

A principally scalar system satisfies certain unique continuation property
\cite{EINT2002}, which will be used as an intermediate step towards the
main theorem. Let $T'>0$ be a positive number, and let $D\subset\mathbb{R}^3$ be
a $C^2$-domain containing the origin. Denote by
$B(0;R)=\{x\in\mathbb{R}^3:|x|<R\}$ the ball of radius $R$ centered at the
origin. We state the following unique continuation result.

\begin{prop}{\cite[Corollary 3.5]{EINT2002}} \label{thm:UC}
Let $u=(u_1,\dots,u_m)$ be a solution of a general principally scalar system
\eqref{pss}. Suppose that there exists $\theta>0$ such that $D\subset B(0;\theta
T')$ and the coefficients $a_j$ in \eqref{pss} satisfy the constraints:
$$
\begin{array}{cl}
\theta^2 a_j (a_j + a_j^{-\frac{1}{2}} |t\nabla a_j|) < a_j +
\frac{1}{2}x\cdot\nabla a_j & \text{ in } (-T',T')\times\overline{D} \\
\theta^2 a_j \leq 1 & \text{ in } \overline{D}.
\end{array}
$$
Then $u=\partial_\nu u=0$ on $(-T',T')\times\partial D$ implies $u=0$ on $\{(t,x)\in(-T',T')\times D:|x|>\theta t\}$.
\end{prop}

The inequality constraints in the theorem justify the
pseudo-convexity of a certain phase function with respect to the wave operators
$\Box_{a_j}$  \cite{EINT2002}. Observe that if $a_j$ are
positive constants, the above constraints reduce to $\theta^2 a_j<1$ in
$\overline{D}$. In this case, any $\theta>0$ that is less than
$\min\{\frac{1}{\sqrt{a_j}}:j=1,\cdots,m\}$ fulfills the inequalities.

Next proposition is the main result of this section. Choose $R>0$ sufficiently
large such that $\Omega\subset B(0;R)$. Recall that $c_{max}$ is defined in
Corollary \ref{prop:FSP}.

\begin{prop} \label{thm:BiotUC}
Let $\mathbf{u}=(\mathbf{u^s}, \mathbf{u^f})$ be the solution of the forward
problem \eqref{forward} with $\mathbf{f}$ compactly supported in $\Omega$.  
Suppose there exists $\theta>0$ such that $B(0;R+c_{max}\frac{3T}{2})\subset
B(0;\theta T)$ and the following inequality constraints hold when $a$ is
replaced by $a_1(x)$, $a_2(x)$, $\mu_1(x)$ respectively:
$$
\begin{array}{cl}
\theta^2 a (a + a^{-\frac{1}{2}} |t\nabla a|) < a + \frac{1}{2}x\cdot\nabla a, & \text{ in } (-\frac{3T}{2},\frac{3T}{2})\times\overline{B(0;R+c_{max}\frac{3T}{2})} \vspace{1ex}\\
\theta^2 a \leq 1 & \text{ in } \overline{B(0;R+c_{max}\frac{3T}{2})}.
\end{array}
$$
If $\mathbf{u^s}(T,x) = \diver\mathbf{u^f}(T,x) = 0 \text{ for } x\in\mathbb{R}^3\backslash\Omega$,
then
$$
\mathbf{u^s}(0,x) = \diver\mathbf{u^f}(0,x) = 0 \quad\text{ for }
x\in\Omega,  x\neq 0,
$$
\end{prop}

\begin{proof}
In view of the initial conditions in \eqref{forward}, we can extend the solution
$\mathbf{u}$ as an even function of $t$ to $(-T,T)\times\Omega$. This extension
is denoted by $\mathbf{u}$ again.

Given $\mathbf{u^s}(T,x)=\diver{\mathbf{u^f}}(T,x)=0$ for
$x\in\mathbb{R}^3\backslash\Omega$, it follows from Corollary \ref{thm:BiotFSP}
that
$$
\mathbf{u^s}(t,x) = \diver{\mathbf{u^f}}(t,x) = 0 \quad\text{ in }
\{(t,x):|x|\geq R+c_{max}|t-T|\}.
$$
As $\mathbf{u}$ is an even function of $t$, we also have
$$
\mathbf{u^s}(t,x) = \diver{\mathbf{u^f}}(t,x) = 0 \quad\text{ in }
\{(t,x):|x|\geq R+c_{max}|t+T|\}.
$$
On the other hand, the fact that $\mathbf{f}$ has compact support in $\Omega$
implies 
$\mathbf{u^s}(0,x)=\diver{\mathbf{u^f}}(0,x)=0$ for
$x\in\mathbb{R}^3\backslash\Omega$. Hence by Corollary \ref{thm:BiotFSP}, we get
$$
\mathbf{u^s}(t,x) = \diver{\mathbf{u^f}}(t,x) = 0 \quad\text{ in }
\{(t,x):|x|\geq R+c_{max}|t|\}.
$$
Combing the above equations gives
$$
\mathbf{u^s}(t,x) = \diver{\mathbf{u^f}}(t,x) = 0 \quad\text{ on }
\{(t,x):|x|=R+c_{max}\frac{T}{2}, -\frac{3T}{2}<t<\frac{3T}{2}\}.
$$
Using Proposition \ref{thm:UC} with the choice $T':=\frac{3T}{2}$,
$D:=B(0,R+c_{max}\frac{T}{2})$ yields
$$
\mathbf{u^s}(0,x) = \diver\mathbf{u^f}(0,x) = 0 \quad \text{ for }
x\in\Omega,  x\neq 0,
$$
which completes the proof.
\end{proof}

\section{Energy Conservation} \label{sec:funcspace}

Define a Sobolev space
$$
H(\diver{};\Omega):=\{\mathbf{v^f}\in (L^2(\Omega))^3: \diver{\mathbf{v^f}}\in
L^2(\Omega)\},
$$
which is equipped with the norm
$$
\|\mathbf{v^f}\|^2_{H(\diver;\Omega)}:=\|\mathbf{v^f}\|^2_{L^2(\Omega)} + \|\diver{\mathbf{v^f}}\|^2_{L^2(\Omega)}.
$$
Consider the space 
$$
V:=\{\mathbf{v}=(\mathbf{v^s},\mathbf{v^f})\in (H^1(\Omega))^3\times
H(\diver;\Omega)\},
$$
which has the norm
$$
\|(\mathbf{v^s}, \mathbf{v^f})\|^2_{V}:=\|\mathbf{v^s}\|^2_{H^1(\Omega)} + \|\mathbf{v^f}\|^2_{H(\diver;\Omega)}.
$$

Introduce a symmetric bilinear form on $V$: 
\begin{align*}
B(\mathbf{v},\mathbf{w})=\int_\Omega
\left[\lambda\diver{\mathbf{v^s}}\cdot\diver{\mathbf{w^s}} + 2 \mu
\left(\epsilon(\mathbf{v^s}) : \epsilon(\mathbf{w}^s)\right) +
r\,\diver{\mathbf{v^f}}\cdot\diver{\mathbf{w^f}} \right] \,dx \\
+ \int_\Omega q \left[ \diver{\mathbf{v^f}}\cdot\diver{\mathbf{w^s}} +
\diver{\mathbf{v^s}}\cdot\diver{\mathbf{w^f}} \right] \,dx,
\end{align*}
where $\epsilon(\mathbf{v^s})=\frac{1}{2}(\nabla\mathbf{v^s} +
(\nabla\mathbf{v^s})^T)$, $A:B={\rm tr}(AB^\top)$ is the Frobenius inner
product of matrices $A$ and $B$.

It is clear to note that $B(\mathbf{v},\mathbf{v})\geq 0$ for all $\mathbf{v}\in
V$ in view of (\textbf{H3}); moreover, $B(\mathbf{v},\mathbf{v})=0$ if and only
if $\mathbf{v^s}=0$ and $\diver{\mathbf{v^f}}=0$. Thus $B$ induces a semi-norm
$\|\mathbf{v}\|_B:=(\mathbf{v},\mathbf{v})_B$. We relate the bilinear form $B$ to the differential operator $P(D)$.

\begin{lemma} \label{thm:BPD}
Suppose $\mathbf{v},\mathbf{w}\in V$ with $\mathbf{v}|_{\partial\Omega}=\mathbf{0}$ or $\mathbf{w}|_{\partial\Omega}=\mathbf{0}$, then
$$
B(\mathbf{v},\mathbf{w}) = (P(D)\mathbf{v},\mathbf{w})_{L^2(\Omega)} = (\mathbf{v},P(D)\mathbf{w})_{L^2(\Omega)}.
$$
\end{lemma}

\begin{proof}
By symmetry, we only need to prove the first equality with the assumption that
$\mathbf{w}|_{\partial\Omega}=\mathbf{0}$. Recalling 
$\mathbf{w}=(\mathbf{w^s},\mathbf{w^f})^T$ and 
\[
P(D)\mathbf{v}=(-\Delta_{\mu,\lambda} \mathbf{v^s} -
\nabla(q\diver{\mathbf{v^f}}), - \nabla(q\diver{\mathbf{v^s}}) -
\nabla(r\diver{\mathbf{v^f}}))^T,
\]
we have
\begin{align} \label{innerprod}
(P\mathbf{v},\mathbf{w})_{L^2(\Omega)} = \int_\Omega \big[&-\Delta_{\mu,\lambda}
\mathbf{v^s} \cdot \mathbf{w^s}- \nabla(q\diver{\mathbf{v^f}}) \cdot
\mathbf{w^s}\notag\\
&- \nabla(q\diver{\mathbf{v^s}}) \cdot \mathbf{w^f} -
\nabla(r\diver{\mathbf{v^f}}) \cdot \mathbf{w^f} \big] \,dx.
\end{align}

To deal with the first integrand, we expand $\Delta_{\mu,\lambda}$ using \eqref{elasticop} to have
$$
\int_\Omega \left[-\Delta_{\mu,\lambda} \mathbf{v^s} \cdot \mathbf{w^s} \right]\,dx = -2 \int_\Omega \diver{(\mu\epsilon(\mathbf{v^s}))}\cdot \mathbf{w^s} \,dx - \int_\Omega \nabla(\lambda\diver{\mathbf{v^s}})\cdot \mathbf{w^s} \,dx.
$$
We claim that
$$
\diver{(\mu\epsilon(\mathbf{v^s}))}\cdot \mathbf{w^s} = \diver{\left(\mu\epsilon(\mathbf{v^s})\right) \mathbf{w^s}} - \mu\epsilon(\mathbf{v^s}):\epsilon(\mathbf{w^s}). 
$$
To justify this, we write
$\epsilon(\mathbf{v^s}):=(\epsilon_1,\epsilon_2,\epsilon_3)^T$, where
$\epsilon_1, \epsilon_2, \epsilon_3$ are the three rows; write
$\mathbf{w^s}=(\mathbf{w^s_1}, \mathbf{w^s_2}, \mathbf{w^s_3})^T$. Then
\begin{align*}
\diver{\left(\mu\epsilon(\mathbf{v^s})\right)}\cdot \mathbf{w^s} & =
\sum^{3}_{j=1} \diver{\left(\mu\epsilon_j\right)}\mathbf{w^s_j} = 
\sum^{3}_{j=1} \left[ \diver{\left(\mu \mathbf{w^s_j}\epsilon_j\right)} -
\mu\epsilon_j\cdot\nabla \mathbf{w^s_j} \right] \\
 & = \diver{\left(\mu\epsilon(\mathbf{v^s})\right) \mathbf{w^s}} - \mu\epsilon(\mathbf{v^s}):\epsilon(\mathbf{w^s}).
\end{align*}
Using the integration by parts yields
\begin{align}
\int_\Omega \left[-\Delta_{\mu,\lambda} \mathbf{v^s} \cdot \mathbf{w^s}\right]\,dx = & \int_\Omega [2 \mu\epsilon(\mathbf{v^s}) : \epsilon(\mathbf{w^s}) + \lambda\diver{\mathbf{v^s}}\cdot\diver{\mathbf{w^s}} ]\,dx \nonumber \vspace{1ex} \\
  & + \int_{\partial \Omega} \left[ -2 \mu \left( \epsilon(\mathbf{v^s})\mathbf{w^s} \right) - (\lambda\diver{\mathbf{v^s}})\mathbf{w^s} \right] \cdot \nu \; dx. \label{bdterm1}
\end{align}
The boundary term vanishes owing to the compact support of $\mathbf{w}$.

The remaining three integrands in \eqref{innerprod} can be treated using
the standard integration by parts:
\begin{align}
 & \int_\Omega \left[ - \nabla(q\diver{\mathbf{v^f}}) \cdot \mathbf{w^s} - \nabla(q\diver{\mathbf{v^s}}) \cdot \mathbf{w^f} - \nabla(r\diver{\mathbf{v^f}}) \cdot \mathbf{w^f} \right] \,dx \nonumber \vspace{1ex} \\
 = & \int_\Omega \left[ q\,\diver{\mathbf{v^f}} \cdot \diver{\mathbf{w^s}} + q\,\diver{\mathbf{v^s}} \cdot \diver{\mathbf{w^f}} + r\,\diver{\mathbf{v^f}} \cdot \diver{\mathbf{w^f}} \right] \,dx \nonumber \vspace{1ex} \\
 & + \int_{\partial\Omega} \left[ -(q \diver{\mathbf{v^f}})\mathbf{w^s} -(q \diver{\mathbf{v^s}})\mathbf{w^f} -(r \diver{\mathbf{v^f}})\mathbf{w^f} \right] \cdot \nu \;dx. \label{bdterm2}
\end{align}
The boundary term again vanishes. This completes the proof.
\end{proof}

Let $L^2(\Omega;M)$ be a weighted $L^2$-space with measure $M(x)dx$ where $M(x)$
is the positive definite matrix defined in \eqref{MPD}; in other words, for any
$\mathbf{v}\in (L^2(\Omega))^6$, $\|\mathbf{v}\|_{L^2(\Omega;M)} :=
(\mathbf{v},M\mathbf{v})_{L^2}$. The space $L^2(\mathbb{R}^3;M)$ and the norm
$\|\cdot\|_{L^2(\mathbb{R}^3;M)}$ is defined similarly with the domain of
integration replaced by $\mathbb{R}^3$.

Given a time-dependent function $\mathbf{u}(t,x)$, define its total energy over
the domain $\Omega$ at time $t$:
\begin{align*}
E_\Omega(t,\mathbf{u})&=  \|\partial_t \mathbf{u}(t,\cdot)\|^2_{L^2(\Omega;M)} +
\|\mathbf{u}(t,\cdot)\|^2_{B} 
 =  \int_\Omega \big(M(x)\partial_t \mathbf{u} \cdot \partial_t \mathbf{u}\\
 &+\lambda |\diver{\mathbf{u^s}}|^2 + 2 \mu |\epsilon(\mathbf{u^s})|^2 +
r |\diver{\mathbf{u^f}}|^2 + 2 q
(\diver{\mathbf{u^f}})(\diver{\mathbf{u^s}}) \big) \; dx.
\end{align*}
This quantity is conservative on a bounded domain $\Omega$ for the solution of
Biot's equations when imposed with appropriate boundary conditions.

\begin{lemma} \label{thm:COE}
Let $\mathbf{u}$ satisfy Biot's equations with zero Dirichlet boundary
condition:
$$ \left\{
\begin{array}{rcl}
M 
\partial^2_t \mathbf{u}
+
P(D) \mathbf{u}
& = & \mathbf{0}  \quad\quad\quad \text{ in } (0,T)\times\Omega, \\
\mathbf{u}|_{[0,T]\times\partial\Omega} &=& \mathbf{0}, \\
\end{array} \right.
$$
then
$$
E_{\Omega}(t,\mathbf{u}) = E_{\Omega}(0,\mathbf{u})\quad\quad 0\leq t \leq
T.
$$
\end{lemma}

\begin{proof}
We briefly sketch the proof since it is similar to that of Lemma
\ref{thm:BPD}. Taking the inner product of the equation with $\partial_t
\mathbf{u} = (\partial_t \mathbf{u^s}, \partial_t \mathbf{u^f})^T$, we obtain
\begin{align*}
0 & =  (\partial^2_t \mathbf{u}, \partial_t \mathbf{u})_{L^2(\Omega;M)} + (P(D)\mathbf{u},\partial_t \mathbf{u})_{L^2(\Omega)} \\
 & =  (\partial^2_t \mathbf{u}, \partial_t \mathbf{u})_{L^2(\Omega;M)} + B(\mathbf{u},\partial_t \mathbf{u}) + b.t.\\
 & =  2 \frac{d}{dt} E_\Omega(t,\mathbf{u}) + b.t..
\end{align*}
Here the second equality is justified by Lemma \ref{thm:BPD}; ``$b.t.$''
represents the arising boundary terms, which are the boundary integrals in
\eqref{bdterm1} and \eqref{bdterm2} with $\mathbf{v}$ and $\mathbf{w}$ replaced
by $\mathbf{u}$ and $\partial_t \mathbf{u}$ respectively. The zero Dirichlet
boundary condition annihilates $b.t.$, which completes the proof of
conservation of energy.
\end{proof}

The proof verifies the well known fact that zero Dirichlet
boundary condition preserves energy. In fact each of the following boundary
conditions
\begin{enumerate}[(i)]
\item $\mathbf{u^s}(t,x)=0$ and $\mathbf{u^f}\cdot\nu (t,x)=0$ on $(0,T)\times\partial\Omega$;
\item $\mathbf{u^s}\cdot\nu (t,x)=0$ and $\mathbf{u^f}(t,x)=0$ on $(0,T)\times\partial\Omega$;
\item $\mathbf{u^s}\cdot\nu (t,x)=0$ and $\mathbf{u^f}\cdot\nu (t,x)=0$ on $(0,T)\times\partial\Omega$.
\end{enumerate}
is energy preserving as well, since each of them is sufficient to annihilate the
boundary term ``$b.t.$'' in the proof.

Let $\mathbf{u}$ be the solution of the direct problem \eqref{forward}. We can
also consider the energy
over the entire space $E_{\mathbb{R}^3}(t,\mathbf{u})$. This global energy is
conservative as well, i.e.,
$$
E_{\mathbb{R}^3}(t,\mathbf{u}) = E_{\mathbb{R}^3}(0,\mathbf{u}) =
\|\mathbf{f}\|^2_B\quad\quad 0\leq t \leq T.
$$
To show this, one can take a large ball $B(0;R)$ so that the solution
$\mathbf{u}(t,\cdot)$ is supported inside $B(0;R)$ for any $t\in [0,T]$. Then
Lemma \ref{thm:COE} is applicable since $\mathbf{u}|_{[0,T]\times\partial
B(0;R)}=\mathbf{0}$ and $E_{\mathbb{R}^3}(t,\mathbf{u}) =
E_{\Omega}(t,\mathbf{u})$ for any $t\in [0,T]$.

\section{Main Theorem} \label{sec:mainthm}

Let $\mathbf{v}$ be the solution of
\begin{equation} \label{invBiot}
\left\{
\begin{array}{rcll}
 M \partial^2_t \mathbf{v} + P(D) \mathbf{v} & = & 0\quad \text{ in }
(0,T)\times\Omega, \\
 \mathbf{v}|_{[0,T]\times\partial\Omega} & = & h, \\
 \mathbf{v}(T,\cdot) & = & \bm{\phi}, \\
 \partial_t \mathbf{v}(T,\cdot) & = & \mathbf{0},
\end{array}
\right.
\end{equation}
where $\bm{\phi}$ is the function satisfying
$$
P(D)\bm{\phi} = 0, \quad\bm{\phi}|_{\partial\Omega} = h(T,\cdot).
$$
The solution $\bm{\phi}$ exists since the analysis in Section \ref{sec:system}
manifests that $P(D)\bm{\phi}=0$ can be transformed into an elliptic
system, which is the time-independent counterpart of \eqref{pssystem}. Define
the time reversal operator
$$
Ah : = \mathbf{v}(0,\cdot).
$$

Now we take $h=\bm{\Lambda} \mathbf{f}$ as the boundary measurement and
expect $A\bm{\Lambda} \mathbf{f}$ to be a reasonable approximation of
$\mathbf{f}$. The rationale, from a microlocal viewpoint, is that the hyperbolic
operator in the forward problem propagates microlocal singularities of
$\mathbf{f}$ to $\partial\Omega$, while the time-reversal process tends to send
back these singularities. This suggests a possible reconstruction of
$\mathbf{f}$, as least on the level of principal symbols.

The microlocal viewpoint also suggests the necessity of an additional assumption
to make sure that all the microlocal singularities of $\mathbf{u}$, the solution
of \eqref{forward}, are not trapped, i.e., all of them are able to reach
$\partial\Omega$ in a finite time. This leads to the non-trapping condition:
there exists a maximal escaping time $T(M, P(D),\Omega)>0$, depending on $M$,
$P(D)$ and the geometry of $\Omega$, such that all the (microlocal)
singularities of $\mathbf{u}$ are out of $\Omega$ whenever $t>T(M,
P(D),\Omega)$; in other words, $\mathbf{u}(t,x)$ is smooth for $x\in\Omega$
whenever $t>T(M, P(D),\Omega)$. Violation of the non-trapping condition would
cause loss of singularities in the measured data $\bm{\Lambda} \mathbf{f}$,
leading to unstable reconstruction of $\mathbf{f}$.

Now we are in the position to state and prove the main theorem. We show
that $A$ is the inverse of $\bm{\Lambda}$ up to a compact operator and the
compact operator becomes a contraction on a suitable function space.

Recall that $\|\cdot\|_B$ is merely a
semi-norm on $V$: it is not positive definite since $\|\mathbf{w}\|_B=0$ implies
only $\mathbf{w^s}=0$ and $\diver{\mathbf{w^f}}=0$. We can take $V$ modulo the
closed subspace $\{\mathbf{w}\in V: \mathbf{w^s}=0, \; \diver{\mathbf{w^f}}=0\}$
to make it a genuine norm. Denote by $B$ the quotient space,
then $(B,\|\cdot\|_B)$ is a Banach space and $\|\mathbf{w}\|_B=0$ implies
$\mathbf{w}=0$ in $B$. We  also project $Ah=\mathbf{v}(0,\cdot)$ to this
quotient space and abuse the notation to call $A$ composed with the canonical
projection as $A$. So $A$ maps into $(B,\|\cdot\|_B)$.

\begin{thm} \label{thm:main}
Let $\Omega$ be non-trapping and $T>T(P(D),\Omega)$. Suppose that the hypotheses
(\textbf{H1})-(\textbf{H3}) and the assumption of Proposition \ref{thm:BiotUC}
are satisfied. Then $K:=I-A\bm{\Lambda}$ is compact and contractive on $B$ in
the sense that $\|K\|_{B\rightarrow B} < 1$. As a consequence, $I-K$ is
invertible on $B$ and
$$
\mathbf{f} = \sum^{\infty}_{j=0} K^j A\bm{\Lambda} \mathbf{f} = A\bm{\Lambda}
\mathbf{f} + K A\bm{\Lambda} \mathbf{f} + K^2 A\bm{\Lambda} \mathbf{f} + \dots.
\quad\quad\quad\text{ in } B.
$$
\end{thm}

\begin{proof}
The proof is divided into two claims. We first show the inequality
$\|K\|_{B\rightarrow B}\leq 1$, and then prove by a contra-positive argument
that the inequality is strict. Given a time-dependent function $g(t,x)$, we
abbreviate $g(t)$ for the spatial function $g(t,\cdot)$.

\emph{Claim 1: $\|K\mathbf{f}\|_B < \|\mathbf{f}\|_B$ unless
$\mathbf{f}=\mathbf{0}$ in $B$.}

Let us give another representation of $K\mathbf{f}$. Let
$\mathbf{u}$ be the solution of \eqref{forward}; let $\mathbf{v}$ be the
solution of \eqref{invBiot} with $h$ replaced by $\bm{\Lambda} \mathbf{f}$.
Denote $\mathbf{w}:=\mathbf{u}-\mathbf{v}$, then $\mathbf{w}$ satisfies
\begin{equation} \label{eqn:diff}
\left\{
\begin{array}{rcll}
 M \partial^2_t \mathbf{w} + P(D) \mathbf{w} & = & \mathbf{0}
\quad\text{ in } (0,T)\times\Omega, \\
 \mathbf{w}|_{[0,T]\times\partial\Omega} & = & \mathbf{0}, \\
 \mathbf{w}(T) & = & \mathbf{u}(T)-\bm{\phi}, \\
 \partial_t \mathbf{w}(T) & = & \partial_t \mathbf{u}(T).
\end{array}
\right.
\end{equation}
Moreover, we have
\begin{equation} \label{K}
K\mathbf{f} = \mathbf{f}-A\bm{\Lambda} \mathbf{f} = \mathbf{u}(0) - \mathbf{v}(0) = \mathbf{w}(0). 
\end{equation}

On the other hand, it is clear to note that
$(\mathbf{u}(T)-\bm{\phi})|_{\partial\Omega}=\mathbf{0}$ by the construction of
$\bm{\phi}$. It follows from Lemma \ref{thm:BPD} that
$$
(\mathbf{u}(T)-\bm{\phi},\bm{\phi})_B =
(\mathbf{u}(T)-\bm{\phi},P(D)\bm{\phi})_{L^2(\Omega)} = 0,
$$
which gives $\|\mathbf{u}(T)-\bm{\phi}\|^2_B = \|\mathbf{u}(T)\|^2_B -
\|\bm{\phi}\|^2_B$. It is easy to verify that 
\begin{align*}
E_\Omega(T,\mathbf{w}) & = \|\partial_t \mathbf{w}(T)\|^2_{L^2(\Omega;M)} + \|\mathbf{w}(T)\|^2_B \\
 & = \|\partial_t \mathbf{u}(T)\|^2_{L^2(\Omega;M)} + \|\mathbf{u}(T)-\bm{\phi}\|^2_B \\
 & = \|\partial_t \mathbf{u}(T)\|^2_{L^2(\Omega;M)} + \|\mathbf{u}(T)\|^2_B - \|\bm{\phi}\|^2_B \\
 & = E_\Omega(T,\mathbf{u}) - \|\bm{\phi}\|^2_B \\
 & \leq E_\Omega(T,\mathbf{u}).
\end{align*}
Combining Lemma \ref{thm:COE} and conservation of energy in
$\mathbb{R}^3$ yields
$$
E_\Omega(0,\mathbf{w}) = E_\Omega(T,\mathbf{w}) \leq E_\Omega(T,\mathbf{u}) \leq E_{\mathbb{R}^3}(T,\mathbf{u}) =  E_{\mathbb{R}^3}(0,\mathbf{u}) = \|\mathbf{f}\|^2_B. 
$$
Thus we have from \eqref{K} that 
\[
\|K\mathbf{f}\|^2_B = \|\mathbf{w}(0)\|^2_B \leq E_\Omega(0,\mathbf{w}) \leq \|\mathbf{f}\|^2_B.
\]
Suppose the equality holds for some $\mathbf{f}\in B$, then all the above inequalities become equalities. In particular $E_\Omega(T,\mathbf{u})=E_{\mathbb{R}^3}(T,\mathbf{u})$, which implies 
$$
\mathbf{u^s}(T,x) = \diver{\mathbf{u^f}}(T,x) = 0 \quad \text{ for }
x\in\mathbb{R}^3\backslash\Omega.
$$
By Proposition \ref{thm:BiotUC}, we obtain $\mathbf{f^s}(x)=0$ and
$\diver{\mathbf{f^f}}(x)=0$ for $x\neq 0$. Changing the value at the single
point $x=0$ does not affect a function in $B$. This completes the proof of Claim
1.

\emph{Claim 2: $K:B\rightarrow B$ is compact and $\|K\|_{B\rightarrow B} <1$.}

Claim 1 alone implies $\|K\|_{B\rightarrow B} \leq 1$. To prove the strict
inequality, we show $K$ is a compact operator on $B$. The spectrum of a
compact operator consists of countably many eigenvalues which may accumulate
only at $0$. Since Claim 1 excludes eigenvalues of modulus 1, the spectral
radius of $K$ must be strictly less than $1$, proving that $\|K\|_{B\rightarrow
B} <1$.

Next we prove the compactness of $K$. Using the representation \eqref{K}, we 
decompose $K$ into composition of bounded operators:
\begin{align*}
&\mathbf{f} \overset{\pi^\ast_1}{\longmapsto} (\mathbf{f},\mathbf{0})
\overset{U_1}{\longmapsto} (\mathbf{u}|_{t=T}-\bm{\phi}, \partial_t
\mathbf{u}|_{t=T})\\
&=(\mathbf{w}(T),\partial_t \mathbf{w}(T))
\overset{U_2}{\longmapsto} (\mathbf{w}(0),\partial_t \mathbf{w}(0))
\overset{\pi_1}{\longmapsto} \mathbf{w}(0).
\end{align*}
Here $\pi_1:(\mathbf{f},\mathbf{g})\mapsto \mathbf{f}$ is the natural projection
onto the first component; $\pi^\ast_1$ is its adjoint; $U_1$ is the solution
operator of the forward problem \eqref{forward}, mapping the state $t=0$ to the
state $t=T$; and $U_2$ is the solution operator of \eqref{eqn:diff} sending
$t=T$ to $t=0$. These are all bounded operators.

Consider
$$
U_1: (\mathbf{f},\mathbf{0}) \longmapsto (\mathbf{u}|_{t=T}-\bm{\phi}, \partial_t \mathbf{u}|_{t=T}).
$$
In view of the assumption that $T>T(P(D),\Omega)$, all the microlocal
singularities of $\mathbf{f}$ have escaped from $\Omega$ at the moment $T$,
hence $(\mathbf{u}|_{t=T},\partial_t \mathbf{u}|_{t=T})$ is a pair of smooth
functions. On the other hand, the function $\bm{\phi}$, as a solution of the
elliptic equations $P(D)\bm{\phi}=0$, is smooth by elliptic regularity. We
conclude that $U_1$ is a smoothing operator, hence compact. This means $K$ is
compact as well, since it is the composition of $U_1$ with other bounded
operators.

We know that $K$ is a contraction on $B$, $(I-K)^{-1}$ exists as a bounded
operator.  Applying $(I-K)^{-1}$ to the identity $(I-K)\mathbf{f} =
A\bm{\Lambda} \mathbf{f}$ and expanding it in terms of Neumann series, we obtain
the reconstruction formula in the statement of the theorem.
\end{proof}

The following stability estimate shows that the faster the energy escapes from
$\Omega$, the faster the convergence of the Neumann series is.

\begin{coro}
Under the assumption of Theorem \ref{thm:main}, the following stability
estimate holds
$$
\|K\mathbf{f}\|_B \leq \left(
\frac{E_\Omega(T,\mathbf{u})}{E_{\Omega}(0,\mathbf{u})} \right)^{\frac{1}{2}}
\|\mathbf{f}\|_B\quad\text{ for  } \mathbf{f}\neq 0 \text{ in } B.
$$
\end{coro}

\begin{proof}
A simple calculation yields that 
$$
\frac{\|K\mathbf{f}\|^2_B}{\|\mathbf{f}\|^2_B} = \frac{\|\mathbf{w}(0)\|^2_B}{E_\Omega(0,\mathbf{u})} \leq \frac{E_\Omega(0,\mathbf{w})}{E_\Omega(0,\mathbf{u})} \leq  \frac{E_\Omega(T,\mathbf{u})}{E_\Omega(0,\mathbf{u})}.
$$
\end{proof}

\textbf{Acknowledgement}. The third author would like to thank Prof. Plamen
Stefanov for many helpful discussions and Prof. Gunther
Uhlmann for bringing the reference \cite{KN2017}. He is also
grateful to the Institute of Computational and Experimenta Research in
Mathematics (ICERM), where part of this research was conducted during his 
participation in the program ``Mathematical Challenges in Radar and Seismic
Imaging'' in 2017.

\end{document}